\documentclass[psamsfonts]{amsart} 
\usepackage{amsmath,amsthm,amscd,amssymb,amsfonts, amsbsy}
\usepackage{graphicx}
\usepackage{latexsym}
\usepackage{txfonts}
\usepackage{exscale}
\usepackage{mathrsfs}
\usepackage{amsfonts}

\calclayout
\allowdisplaybreaks

\usepackage{color}

\usepackage{pdfsync}
\usepackage[colorlinks,linkcolor=blue,citecolor=blue,pagebackref,hypertexnames=false, breaklinks]{hyperref}

\newtheorem{theorem}{Theorem}[section]
\newtheorem{proposition}[theorem]{Proposition}
\newtheorem{coro}[theorem]{Corollary}

\newtheorem{rem}[theorem]{Remark}

\renewcommand{\epsilon}{\varepsilon}

\newcommand\R{\mathbb{R}}

\newcommand{\norm}[1]{{\left\lVert{#1}\right\rVert}}

\allowdisplaybreaks[4]
\newcommand\restr[2]{{
  \left.\kern-\nulldelimiterspace 
  #1 
  \vphantom{\big|} 
  \right|_{#2} 
  }}

\numberwithin{equation}{section}

\begin{document}
\title{A note on Sobolev type  inequalities on graphs with polynomial volume growth}
\author{Li Chen}
\address{Department of Mathematics, University of Connecticut, Storrs, CT 06269}
\email{}

\begin{abstract} 
We prove the generalized $L^p$-Poincar\'e inequalities and Sobolev type inequalities on graphs with polynomial volume growth. They are optimal on Vicsek graphs.
\end{abstract}

\maketitle

\section{Introduction}
The classical Sobolev inequalities play a crucial role in analysis and partial differential equations in Euclidean spaces. In many different settings, Sobolev inequalities have been deeply explored and they are extremely flexible and  useful. For a nice introduction of previous work, we refer to the survey by L. Saloff-Coste \cite{SC09}, see also the comprehensive book \cite{SC02}. 

In this note we work on a scale of Sobolev type inequalities of great generality in the setting of graphs, which was first introduced by T. Coulhon in \cite{C96} on manifolds and graphs. Let $D>1$ and $1\le p\le \infty$, consider on an infinite graph $\Gamma$ the following scale of  inequalities
\begin{equation}\label{eq:SobD}
\Vert f\Vert_{p} \leq |\Omega|^{1/D}\,\Vert  \,|\nabla f|\,\Vert_{p},
\end{equation}
for any finite subset $\Omega$ of $\Gamma$ and any function $f$ supported in $\Omega$, where $|\Omega|$ denotes the measure of $\Omega$. 

The above inequalities are  also quoted as Faber-Krahn type inequality (see for instance \cite[page 88]{C96}). When $p<D$, the inequality \eqref{eq:SobD} is equivalent to the graph analogue of classical Sobolev inequalities. For more  interpretations from either geometric or analytic aspects, we refer to \cite{C96,C95,B17}. Of the same spirit, a more general version of \eqref{eq:SobD} was also studied in the context of manifolds and graphs in \cite{C95,C03, CHS}, that is,
\begin{equation}\label{eq:SobVarphi}
\Vert f\Vert_{p} \leq \varphi(|\Omega|)\,\Vert  \,|\nabla f|\,\Vert_{p},
\end{equation}
 where  $\varphi$ is a non-decreasing function on $\mathbb R^+$ to itself. 

Volume lower bounds are important implications of the above two scales. Indeed, \eqref{eq:SobD}  implies that $|B(x,n)|\ge Cn^D$. For the more general scale, if $\varphi$ is unbounded, then  \eqref{eq:SobVarphi} for $p=\infty$ is equivalent to $|B(x,n)|\ge C\varphi^{-1}(n)$, where $\varphi^{-1}$ is the reverse function of $\varphi$ (see \cite[Proposition 2.2]{C03} or \cite{C95}). 
The reverse implication is well studied under the assumption of pseudo-Poincar\'e inequality
\begin{equation}\label{PP_p}
\|f-f_n\|_p\le Cn\|\,|\nabla f|\,\|_p,
\end{equation}
where $f_n(x)=\frac1{|B(x,n)|}\sum_{y\in B(x,n)}f(y)\mu(y)$, where $\mu$ is the measure on $\Gamma$. In fact, this allows the scale to go down from $\infty$ to $p$, see \cite[Proposition 2.6]{C03}. 
Compared with the Poincar\'e inequality
\begin{equation}\label{P_p}
\|f-f_{B(x,n)}\|_{L^p(B(x,n))} \le Cn\|\,|\nabla f|\,\|_{L^p(B(x,2n))},
\end{equation}
the pseudo-Poincar\'e inequality may hold for more general situations, for instance on all unimodular Lie groups (see \cite{C03}). As is well-known, the  Poincar\'e inequality and the volume doubling property imply the pseudo-Poincar\'e inequality (see for example \cite{SC92}, \cite{SC02}). On graphs both the Poincar\'e and pseudo-Poincar\'e inequalities are very restrictive properties. However, under assumptions of volume growth,  one can obtain analogues of the $L^1$ and $L^2$-Poincar\'e inequalities, see \cite[Section 5]{CSC93}. 

In this note,  our goals are to generalize the results in \cite{CSC93} to the case $1<p<\infty$  on graphs with polynomial volume growth and hence to study the scale of Sobolev type inequalities in form of \eqref{eq:SobVarphi}. We are particularly interested in explicit examples of Vicsek graphs (see \cite{BCG,B04}).

In the following we introduce the setting of graphs, mainly following the notation in \cite{B17}. Let $\Gamma=(V,E) $ be a connected undirected infinite graph with the set of vertices $V$ and the set of edges $E$. Endow $\Gamma$ with a symmetric weight $\mu_{xy}$, $x,y\in V$ such that $\mu_{xy}=\mu_{yx}\geq 0 $. We say that $x$ and $y$ are neighbors, i.e., $\{x,y\}\in E$, denoted by $x\sim y$, if and only if $\mu_{xy}>0$.

Define $\mu (x)=\sum\limits_{y\sim x}\mu_{xy} $, then it extends to a measure on $V$ by 
\[
\mu(\Omega )=\sum_{x\in \Omega }\mu_x,
\]
where $\Omega$ is a finite subset in $V$. 

For $x,y\in V$, a path of length $n$ between $x$ and $y$  is a sequence $x_0,x_1,\cdots,x_n$ with $x_0 = x, x_n = y$ and $x_{i-1}\sim x_i, 1\le i\le n$.  The induced graph distance $d(x,y)$ is the minimal number of edges in any path connecting $x$ and $y$. The graph $\Gamma$ is connected if $d(x,y)<\infty$ for all $x,y\in V$.
Define balls in $\Gamma$ by 
\[
B(x,r)=\{y\in V: d(x,y)\leq r, x\in V, r\ge 0\}.
\] 

Any weighted graph $(\Gamma,\mu) $ admits a random walk on $\Gamma$ defined by the transition probabilities $p(x,y)=\frac{\mu_{xy}}{\mu (x)}$ if $x\sim y$ and otherwise $0$. Then $p$ is a reversible Markov kernel satisfying
\begin{eqnarray*}
\mu(x) p(x,y)=\mu_{xy}=\mu(y) p(y,x),
\quad
\sum_{y}p(x,y)=1.
\end{eqnarray*}

Denote by $\mathcal C(V)=\{f:V\to \R\}$ and by  $\mathcal C_0(V)$ the set of functions in $\mathcal C(V)$ with finite support.
For $p\in[1,\infty)$, the $L^p$ norm of a function $f\in \mathcal C(V)$ is given by
\[
\Vert f\Vert _p=\left(\sum_x |f(x)|^p \mu(x) \right)^{1/p}, 
\]

The linear operator $P$ associated with the kernel $p$ is defined by 
\[
Pf(x)=\sum_y p(x,y)f(y).
\]
We call the operator $I-P$ the (probabilistic) Laplacian on $\Gamma$.

The value of the discrete gradient on $\Gamma $ is defined by
\[
|\nabla f(x)|=\left(\frac{1}{2}\sum_y p(x,y)\,|f(x)-f(y)|^2\right)^{1/2}.
\]
Then we have that $\langle(I-P)f,f\rangle=\| \,|\nabla f|\,\|_{2}^2$, where $\langle\cdot,\cdot\rangle$ denotes the inner product on $L^2(V,\mu)$.

Throughout this paper, we always assume that 
\begin{itemize}
\item  $\Gamma $ is locally finite, i.e., there exists $N\in \mathbb N $ such that any $x\in V $ has at 
most $N$ neighbors;

\item $\mu(x)\ge c_1$ and $(\Gamma,\mu)$ has controlled weights, i.e., there exists $c_2>0$ such that for all $x,y\in V $
\begin{equation}\label{eq:controlled weights}
x\sim y \text{ implies } p(x,y)\geq c_2;
\end{equation}

\item The volume growth of $(\Gamma,\mu)$ is uniformly polynomial. That is, there exists $D>0$ such that $\mu(B(x,n))\simeq n^D$.
\end{itemize}

Note that under these assumptions, we have that $\mu_{xy}>c$ and for $1\leq p<\infty$,
\begin{equation}\label{grad}
\sum_x |\nabla f(x)|^p \mu(x) \simeq \sum_{x,y} |f(x)-f(y)|^p \mu_{xy},
\end{equation}
see, for example, \cite[p. 313]{BR09} for a proof.

Our main results are as follows.
\begin{theorem}\label{thm:pD}
Let $(\Gamma,\mu)$ be a locally uniformly finite weighted graph with controlled weight. Assume that $\mu(B(x,n))\simeq n^D, \forall n\in \mathbb N$. Then for $p \geq 1$, it holds  that 
\begin{equation}\label{pp}
\Vert f-f_n\Vert_{L^p(B(x,n))} \leq C n^{\frac{D}{p}+\frac{1}{p'}}\Vert  \,|\nabla f|\,\Vert_{L^p(B(x,2n))}.\tag{$P_p^{D}$}
\end{equation}
\end{theorem}

\begin{theorem}\label{thm:Sobolev}
Let $(\Gamma,\mu)$ be a locally uniformly finite weighted graph with controlled weight and let $\Omega\subset \Gamma$ be a bounded subset. Assume that $\mu(B(x,n))\simeq n^D,\forall n\in \mathbb N$. Then for $p \geq 1$, it holds 
\begin{equation}\label{Sobolev}
\Vert f\Vert_{p} \le C \mu(\Omega)^{\frac{1}{p}+\frac{1}{p'D}} \Vert \,|\nabla f|\, \Vert_{p},\quad \forall f \in \mathcal C_0 (\Omega).\tag{$S_p^{D}$}
\end{equation}
\end{theorem}

Note that when $p=1$, the above inequality \eqref{Sobolev} is trivial since it  is equivalent to $\mu(\partial \Omega)\ge C$ following the co-area formula (see for instance \cite[Proposition 2.3]{C03}).

Define the escape time $T(x,r)$ to be the mean exit time of a simple random walk on $\Gamma$ starting at $x$ from the ball with center $x$ and radius $r$. We say $\Gamma$ has escape time exponent $\beta>0$ if $T(x, r) \simeq r^{\beta}$ for $r\ge1$. Well known estimates for random walks on graphs imply that $D\ge 1$ and $2 \le \beta\le D+1$. In particular, Vicsek graphs are borderline examples with $\beta=D+1$ (see \cite{B04}), which are of particular interest to us. In fact, we have
\begin{proposition}\label{opt}
Let $\Gamma$ be a Vicsek graph. Then the Poincar\'e inequality \eqref{pp} and Sobolev type inequality \eqref{Sobolev} in Theorems 1.1 and 1.2 are optimal for $1\leq p\leq \infty$.
\end{proposition}

\begin{rem}
In Theorem \ref{thm:pD} and Theorem \ref{thm:Sobolev}, we have variant parameters in inequalities  \eqref{pp} and \eqref{Sobolev}, which depends on $p$ and $D$. When these inequalities are optimal, then the related parameters can't be improved to smaller numbers.
\end{rem}

The paper is organized as follows. In Section 2, we give the proof of Theorem \ref{thm:pD}. Section 3 provides two different proofs of Theorem \ref{thm:Sobolev}. Finally, we study the special case of Vicsek graphs and prove Proposition \ref{opt}.
\bigskip

\section{Poincar\'e  type inequalities}
In this section, we will prove Theorem \ref{thm:pD}. Recall that under assumptions of volume growth, the $L^1$ and $L^2$-Poincar\'e type inequalities were partially obtained in \cite{CSC93}.  In fact, if the the volume growth is uniformly polynomial, i.e.,  $\mu(B(x,n))\simeq n^D$, then for any $f\in \mathcal C_0(V)$,
\[
\norm{f-f_{B(x,n)}}_{L^1(B(x,n))} \leq C n^{D}\Vert  \,|\nabla f|\,\Vert_{L^1(B(x,2n))}, \,\,\forall x\in V,
\]
and 
\[
\norm{f-f_{B(x,n)}}_{L^2(B(x,n))} \leq C n^{\frac{D+1}{2}}\Vert  \,|\nabla f|\,\Vert_{L^2(B(x,2n))}, \,\,\forall x\in V.
\]
We will extend ideas from \cite{CSC93} (see also \cite{DS}) to prove our first  result.

\begin{proof}[Proof of Theorem \ref{thm:pD}]
For any $y,z \in B(x,n)$, we choose $\gamma_{y,z}$ as one of the shortest paths from $y$ to $z$. Define
$\Gamma_{x,n} :=\{\gamma_{y,z}: y,z \in B(x,n)\}$. Let $e$ be an oriented edge in a path from $e_-$ to $e_+$ and $\mu_e=
\mu_{e_- e_+}$.

Since
\[
|f(y)-f(z)| \leq \sum_{e \in \gamma_{y,z}}|f(e_+)-f(e_-)|,
\]
for $p=1$, we have
\begin{align*}
 \mu(B(x,n))\sum_{y \in B(x,n)}|f(y)-f_{B(x,n)}| \mu(y)
 & \leq  
\sum_{y,z \in B(x,n)}|f(y)-f(z)|\mu(y)\mu(z)
\\ & \leq  
\sum_{y,z \in B(x,n)}\sum_{e \in \gamma_{y,z}}|f(e_+)-f(e_-)|\,\frac{\mu_e}{ \mu_e}\mu(y)\mu(z)
\\ & \leq  
C \sum_{y,z \in B(x,n)} \left(\sum_{e \in \gamma_{y,z}}|f(e_+)-f(e_-)| \mu_e\right) \mu(y)\mu(z).
\end{align*}
In the last inequality, $C$ depends on the weight $\mu$.

For $p>1$, the H\"older inequality leads to
\begin{align*}
 \mu(B(x,n)) \sum_{y \in B(x,n)} |f(y)-f_{B(x,n)}|^p \mu(y)
 & \leq 
 \mu(B(x,n))^{1-p} \sum_{y \in B(x,n)} \left |\sum_{z \in B(x,n)} [f(y)-f(z)]\mu(z)^{\frac{1}{p}+\frac{1}{p'}}\right |^p \mu(y)
\\ & \leq  
 \mu(B(x,n))^{1-p} \sum_{y \in B(x,n)} \left(\sum_{z \in B(x,n)} |f(y)-f(z)|^p \mu(z)\right)
\cdot \left(\sum_{z \in B(x,n)} \mu(z)\right)^{p-1}\mu(y)
\\ & =  
\sum_{y,z \in B(x,n)}|f(y)-f(z)|^p \mu(y)\mu(z)
\\ & \leq  
\sum_{y,z \in B(x,n)}\left(\sum_{e \in \gamma_{y,z}}|f(e_+)-f(e_-)|\,\frac{\mu_e^{1/p}}{\mu_e^{1/p}}\right)^p \mu(y)\mu(z)
\\ & \leq  
\sum_{y,z \in B(x,n)} \left(\sum_{e \in \gamma_{y,z}}|f(e_+)-f(e_-)|^p \mu_e\right)
\cdot \left(\sum_{e \in \gamma_{y,z}} \mu_e^{-1/(p-1)}\right)^{p-1} \mu(y)\mu(z).
\end{align*}

Since $\Gamma$ has controlled weights, the last term is bounded from above by 
\[
C \sum_{e \in B(x,2n)}|f(e_+)-f(e_-)|^p \mu_e \cdot \sum_{y,z \in B(x,n)} |\gamma_{y,z}|^{p-1} \mu(y)\mu(z).
\]
Take $K(x,n)=\sum\limits_{y,z \in B(x,n)} |\gamma_{y,z}|^{p-1} \mu(y)\mu(z)$ for $p \geq 1$, then we get from (\ref{grad})
\[
\Vert f-f_{B(x,n)}\Vert_{L^p (B(x,n))}^p \leq C K(x,n) \Vert\,|\nabla f|\,\Vert_{L^p (B(x,2n))}^p,\,
\forall p \geq 1.
\]
One can see that $K(x,n) \leq (2n)^{p-1} V(x,n)^2$. Therefore,
\[
\Vert f-f_{B(x,n)}\Vert_{L^p (B(x,n))}^p \leq C  \mu(B(x,n))^{1+\frac{p-1}{D}} \Vert \,|\nabla f|\, \Vert_{L^p (B(x,2n))}^p.
\]
Since $ \mu(B(x,n)) \simeq n^D$, we have
\[
\Vert f-f_{B(x,n)}\Vert_{L^p (B(x,n))}^p \leq C n^{D+p-1} \Vert \,|\nabla f|\,\Vert_{L^p (B(x,2n))}^p.
\]
This leads to \eqref{pp} and hence we finish the proof.
\end{proof}

The pseudo-Poincar\'e type inequality follows now as a corollary.
\begin{coro}
Let $(\Gamma,\mu)$ be a locally uniformly finite weighted graph with controlled weight. Assume that $\mu(B(x,n))\simeq n^D, \forall n\in \mathbb N$. Then for $p \geq 1$, it holds  that 
\begin{equation}\label{vicsek-pp}
\Vert f-f_n\Vert_{p} \leq C n^{\frac{D}{p}+\frac{1}{p'}} \Vert \,|\nabla f|\, \Vert_{p }. \tag{$PP_p^{D}$}
\end{equation}
\end{coro}
The proof is the similar to  the one that Poincar\'e inequality together with the volume doubling  property implies the pseudo-Poincar\'e inequality, see for instance \cite[Lemma 5.3.2]{SC02}. We leave details to the interested reader. 

\section{Sobolev type inequalities }
In this section, we will prove Theorem \ref{thm:Sobolev}, for which we provide two different approaches.
One is to use the Poincar\'e type inequality \eqref{vicsek-pp}, following  \cite[Section III]{C95} (see also \cite[Proposition 2.6]{C03}).

\begin{proof}[Proof of Theorem \ref{thm:Sobolev}]
We first treat the case $p>1$. Note that for any $f\ge 0$ in $\mathcal C_0 
(\Omega)$, there holds for any $x\in V$ and $n\ge 1$
\begin{equation}\label{ineq}
|f_n(x)| =\left| \frac{1}{ \mu(B(x,n))} \sum_{y\in B(x,n)} f(y)\mu(y)\right| \leq C_D n^{-D} \Vert f\Vert_{1}.
\end{equation}
 Write
\[
\Vert f\Vert_{p}^p= (f-f_n, f^{p-1})+(f_n, f^{p-1}).
\]
It follows from H\"older inequalities, (\ref{vicsek-pp}), and (\ref{ineq}) that 
\begin{eqnarray*}
\Vert f\Vert_{p}^p
& \leq & 
\Vert f-f_n\Vert_{p}\cdot \Vert f\Vert_{p}^{p-1}+\Vert f_n\Vert_{\infty} \cdot \Vert f^{p-1}\Vert_{1}
\\ & \leq &
C n^{\frac{D}{p}+\frac{1}{p'}} \Vert\,|\nabla f|\,\Vert_{p}\cdot \Vert f\Vert_{p}^{p-1}+
\frac{C_D \mu(\Omega)}{n^{D}}\Vert f \Vert_{p}^p.
\end{eqnarray*}
Choosing $n$ to be the integer such that $n^D \simeq 2 C_D\mu(\Omega)$, then 
\[
\Vert f\Vert_{p} \leq C\mu(\Omega)^{\frac{1}{p}+\frac{1}{p'D}}\Vert \,|\nabla f|\, \Vert_{p}.
\]
The inequality also holds for any $f \in \mathcal C_0 (\Omega)$.

Next let $p=1$. For any $\lambda>0$, write 
\[
\mu(\{|f|\ge\lambda\}) \le \mu(\{|f-f_n|>\lambda/2\})+\mu(\{|f_n|>\lambda/2\}).
\]
Taking $n\simeq \left(2\|f\|_1/\lambda\right)^{1/D}$ to be the integer such that $\{|f_n|>\lambda/2\}=\emptyset$. Therefore by $(PP_1^D)$,
\[
\mu(\{|f|\ge\lambda\}) \le \frac{2}{\lambda} \|f-f_n\|_1 \le \frac{2}{\lambda} n^D \|\,|\nabla f|\,\|_1\le C\|f\|_1\|\,|\nabla f|\,\|_1.
\]
Taking $\lambda=1$ and $f=\mathbf 1_{\Omega}$ yields $\mu(\partial \Omega)\ge C$, which is equivalent to $(S_1^D)$. 
\end{proof}

The second approach, without using the Poincar\'e inequality, is  inspired by the Faber-Krahn inequality in \cite{BCG}. That is,  for a graph $(\Gamma,\mu)$ with controlled weights, the Faber-Krahn inequality as follows holds for any non-empty finite set $\Omega\subset \Gamma$,
\[
\lambda_1(\Omega):=\inf_{f\in \mathcal C_0(\Omega)}\frac{\sum\limits_{x\in \Omega } |\nabla f(x)|^2 \mu (x) }{\sum\limits_{x\in \Omega } |f(x)|^2 \mu (x)} \geq \frac{c}{r(\Omega)\mu(\Omega)},
\]
where
\[
r(\Omega) = \max \{ r \in \mathbb N: \exists x \in \Omega \text{ such that } B(x,r) \subset \Omega \}.
\]
If $\Gamma$ has polynomial volume upper bound $\mu(B(x,r))\gtrsim r^D$, the Faber-Krahn inequality is also the $L^2$-Sobolev type inequality in our theorme. Moreover, this inequality is sharp at all scales of the volume on Vicsek graphs (\cite[Theorem 4.1]{BCG}).

\begin{proof}[An alternative proof of Theorem \ref{thm:Sobolev}] 
For any function $f \in \mathcal C_0(\Omega)$, assume $\max\limits_{x \in \Omega} |f(x)| = 1$ (otherwise we can normalize $f$). Then we have
\begin{equation}\label{f}
\sum_{x\in \Omega } |f(x)|^p \mu (x) \leq \mu(\Omega).
\end{equation}

Consider a point $x_0$ such that $| f(x_0)| = 1$ and the largest integer $n$ such that the ball $B(x_0, n) \subset \Omega$. Obviously we have $n \leq r(\Omega)$. Also, there exists a sequence of points
\[
x_0 \sim x_1 \sim x_2 \sim \cdots \sim x_n \sim x_{n+1}
\]
starting from $x_0$ and terminating at a point $x_{n+1} \notin \Omega$.

Therefore, by \eqref{grad} and the H\"older inequality
\begin{equation}\label{df}
\begin{split}
\sum_{x\in \Omega } |\nabla f(x)|^p \mu (x) 
&\simeq 
\sum_{x,y \in \Omega } |f(x)-f(y)|^p \mu_{xy}
\\ &\geq \frac{c}{n^{p-1}}\left( \sum_{i=0}^{n}|f(x_i)-f(x_{i+1})|\right)^p
\geq \frac{c}{n^{p-1}},
\end{split}
\end{equation}
where the last inequality is due to the fact that
\[
\sum_{i=0}^{n}|f(x_i)-f(x_{i+1})| \geq |f(x_0)-f(x_{n+1})|=1.
\]
It follows from (\ref{f}) and (\ref{df}) that
\[
\frac{\sum\limits_{x\in \Omega } |\nabla f(x)|^p \mu (x) }{\sum\limits_{x\in \Omega } |f(x)|^p \mu (x)}
\geq \frac{c}{n^{p-1} \mu(\Omega)}
\geq \frac{c}{\mu(\Omega)^{1+(p-1)/D}}.
\]
Eventually, we obtain \eqref{Sobolev} and the proof is completed.

\end{proof}

\bigskip

\section{Optimality on Vicsek graphs}
Considering the Vicsek graphs, we assume that the weight $\mu$ is the standard weight. Recall the construction of a Vicsek graph $\Gamma$ on $\mathbb R^d$ (taken from \cite{CG03}):
Let $Q_r$ denote the cube in $\mathbb R^n$
\[
Q_r=\{x \in \mathbb R^d: 0\leq x_i\leq r,\,\, i=1,2,\cdots,d\}.
\]
Construct an increasing sequence $\{\Gamma_k\}$ of finite graphs as subsets of $Q_{3^k}$. Let $\Gamma_1$ be the set of 
$2^d +1$ points containing all vertices of $Q_1$ and the center of $Q_1$. Define $2^d$ edges in $\Gamma_1$ as segments 
connecting the center with the corners. Assuming that $\Gamma_k$ is already constructed, define $\Gamma_{k+1}$ as follows. 
The cube $Q_{3^k+1}$ is naturally divided into $3^n$ congruent copies of $Q_{3^k}$; select $2^d +1$ of the copies of 
$Q_{3^k}$ by taking the corner cubes and the center one. In each of the selected copies of $Q_{3^k}$ construct a congruent 
copy of graph $\Gamma_k$, and define $\Gamma_{k+1}$ as the union of all $2^d+1$ copies of $\Gamma_k$ (merged at the 
corners). Then the Vicsek tree $\Gamma$ is the union of all $\Gamma_k$, $k\geq 1$ (see Figure \ref{fig1}). 
\begin{figure}
\begin{center}
\includegraphics[scale=0.8]{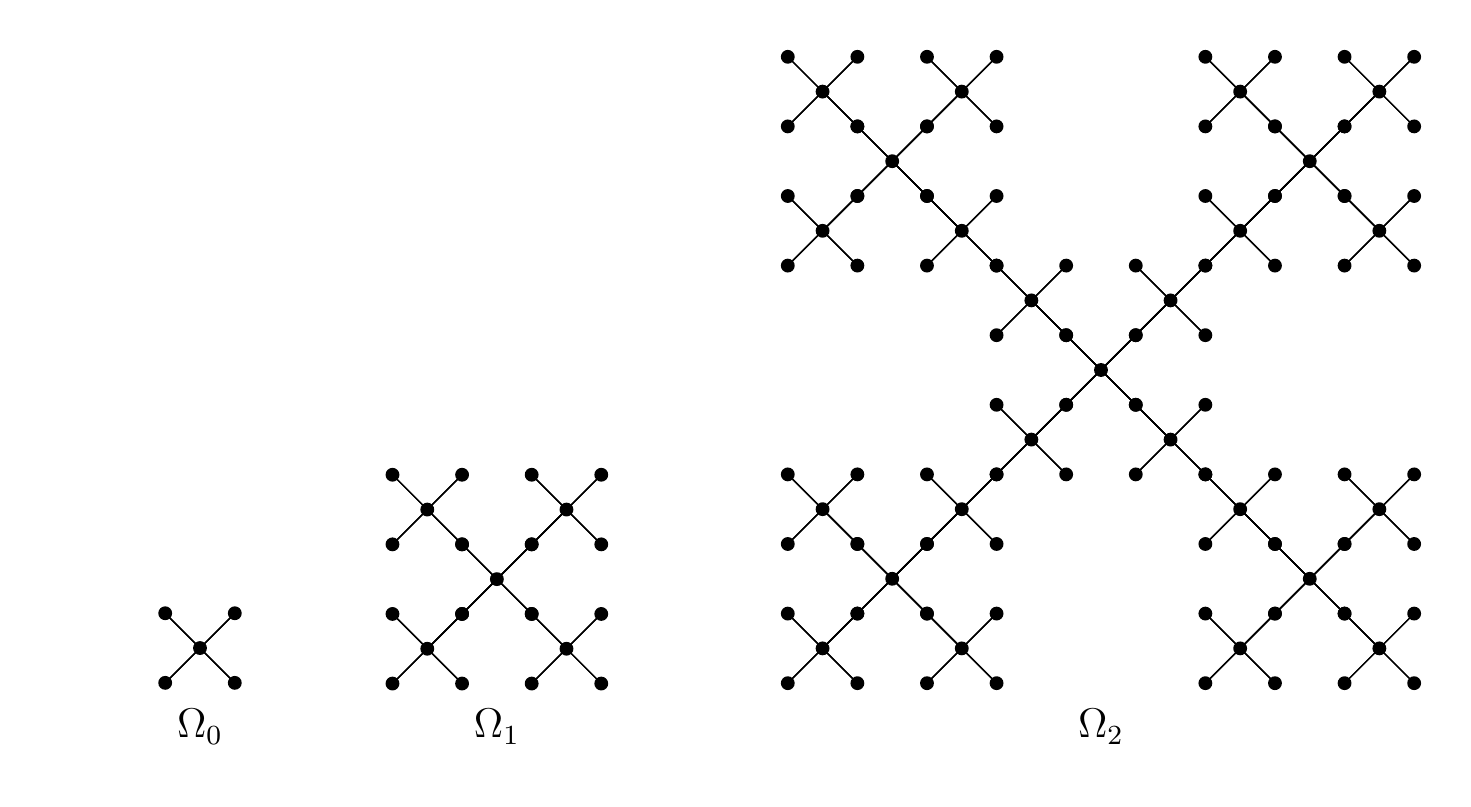}
\end{center}
\caption{The first steps of the Vicsek graph with parameter $D=\log_3 5$  built in $\R^2$.}\label{fig1}
\end{figure}
Consider the natural weight $\mu_{xy}=1$ if $x\sim y$ and otherwise $0$.
For all $x\in \Gamma$, $r\geq 1$ and $k>0$, $\Gamma$ satisfies
\[
\mu(B(x,r))\simeq r^D
\]
and 
\[
p_k(x,x) \simeq C k^{-\frac{D}{D+1}}.
\]
where $D=\log_3 (2^d+1)$. By Theorem \ref{thm:pD}, the Vicsek graph satisfies the following scale-invariant Poincar\'e inequality: for any $f\in \mathcal C_0(\Gamma)$,
\begin{equation}\label{P_2'}
\norm{f-f_{B(x,r)}}_{L^2(B(x,n))}  \leq C r^{\frac{D+1}{2}} \norm{ \,|\nabla f|\,}_{L^2(B(x,2r))}, \,\,\forall x\in V,
\end{equation} 
Since $D>1$, this inequality is strictly weaker than ($P_2$).

\begin{proof}[Proof of Proposition \ref{opt}] It suffices to show the optimality of the generalized Sobolev inequality \eqref{Sobolev} for $p\geq 1$. Indeed,  note that we can obtain \eqref{Sobolev} from \eqref{pp} in the proof of Theorem 
1.2. If the exponent $D(p)=\frac{D}{p}+\frac{1}{p'}$ of $r$ in \eqref{pp} is not optimal, then \eqref{Sobolev} can also be improved, which contradicts its optimality.

In fact, with the same functions and subsets, we can also show the optimality of \eqref{Sobolev} for $p\geq 1$.  

Let $\Omega_n=\Gamma \bigcap [0,3^n]^d$ be the same subset as in 
\cite{BCG}, where $q=2^d+1=3^D$. Hence $\mu(\Omega_n) \simeq q^n$. Denote by $z_0$ the centre of $\Omega_n$ and by 
$z_i , i \geq 1$ its corners. Define $F_n$ as follows: $F_n(z_0) = 1, F_n(z_i) = 0, i \geq 1$, and extend $F_n$ as a harmonic function in the 
rest of $\Omega_n$. Then $F_n$ is linear on each of the paths of length $3^n$, which connects $z_0$ with the corners $z_i$, and is constant elsewhere. More exactly, if $z$ belongs to some $\gamma_{z_0,z_i}$, then $F_n(z)=3^{-n}d(z_i,z)$. If not, then $F_n(z)=F_n(z')$, where $z'$ is the nearest vertex in certain line of  $z_0$ and $z_i$. See Figure \ref{fig2}, which we take from \cite{BCG}.
\begin{figure}
\begin{center}
\includegraphics[scale=0.8]{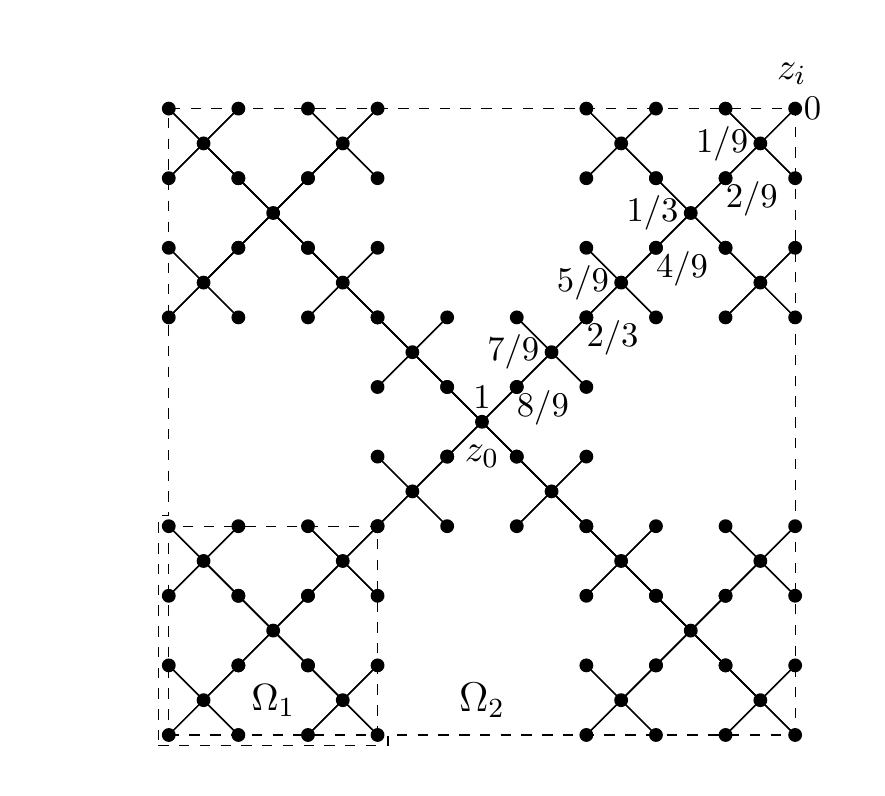}
\end{center}
\caption{the function $g_2$ on the diagonal $z_0 z_i$} \label{fig2}
\end{figure}

For any $x$ in the $n-1$  block with centre $z_0$, we have  $F_n(x) \geq \frac{2}{3}$. Therefore
\[
\sum_{x\in \Omega_n} |F_n(x)|^p \mu(x) \geq (2/3)^p \mu(\Omega_{n-1}) \simeq \mu(\Omega_{n}).
\]

Also, since $|F_n(x)-F_n(y)| = 3^{-n}$ for any two neighbours $x, y$ on each of the diagonals connecting $z_0$ and $z_i$, and otherwise 
$F_n(x)-F_n(y)= 0$, we obtain
\begin{eqnarray*}
\Vert \,|\nabla F_n|\,\Vert_{p}^p &\simeq& \sum_{x,y \in \Omega_n}|F_n(x)-F_n(y)|^p \mu_{xy}
\leq \sum_{i=1}^{2^N} 3^{-np} d(z_0,z_i)=2^N 3^{-n(p-1)}
\\ &\simeq& \mu(\Omega_n)^{-\frac{p-1}{D}}.
\end{eqnarray*}

Finally combining the above two estimates, we have
\[
\frac{\Vert\,|\nabla F_n|\,\Vert_{p}}{\Vert F_n\Vert_{p}} \lesssim  \mu(\Omega_n)^{-\frac{1}{p'D}-\frac{1}{p}}.
\]
This finishes the proof.
\end{proof}

\begin{rem}
In \cite{CCFR}, it was proved that $\left\|(I-P)^{1/2}f\right\|_p \le C\|\,|\nabla f|\,\|_p$ doesn't hold for $1<p<2$ and hence by duality the Riesz transform $\nabla(I-P)^{-1/2}$ is not bounded on $L^p$ for $p>2$. This result is strikingly different from Euclidean spaces and its proof relies on an argument of contradiction for which we used the same family of functions as in the above proof. It would be of further interest to explore the intrinsic role that the weaker Poincar\'e inequalities play on the boundedness of the Riesz transform. 
\end{rem}

\bibliographystyle{plain}

\end{document}